\documentclass[12pt]{article}
\usepackage{amssymb}
\usepackage{amsthm}
\usepackage{amsmath}
\usepackage{graphicx}
\usepackage{enumerate}

\newtheorem{theorem}{Theorem}[section]
\newtheorem{definition}[theorem]{Definition}

\newtheorem{remark}[theorem]{Remark}
\newtheorem{example}[theorem]{Example}
\newtheorem{corollary}[theorem]{Corollary}
\title{On ring-like structures of lattice-ordered numerical events} 
\author{Dietmar Dorninger and Helmut L\"anger}
\date{}
\begin{document}

\footnotetext{Support of the research of the second author by the Austrian Science Fund (FWF), project I~4579-N, entitled ``The many facets of orthomodularity'', as well as by \"OAD, project CZ~02/2019, entitled ``Function algebras and ordered structures related to logic and data fusion'', is gratefully acknowledged.}

\maketitle

\begin{abstract} 
Let $S$ be a set of states of a physical system. The probabilities $p(s)$ of the occurrence of an event when the system is in different states $s\in S$ define a function from $S$ to $[0,1]$ called a numerical event or, more accurately, an {\em $S$-probability}. Sets of $S$-probabilities ordered by the partial order of functions give rise to so called algebras of $S$-probabilities, in particular to the ones that are lattice-ordered. Among these there are the $\sigma$-algebras known from probability theory and the Hilbert-space logics which are important in quantum-mechanics. Any algebra of $S$-probabilities can serve as a quantum-logic, and it is of special interest when this logic turns out to be a Boolean algebra because then the observed physical system will be classical. Boolean algebras are in one-to-one correspondence to Boolean rings, and the question arises to find an analogue correspondence for lattice-ordered algebras of $S$-probabilities generalizing the correspondence between Boolean algebras and Boolean rings. We answer this question by defining ring-like structures of events (RLSEs). First, the structure of RLSEs is revealed and Boolean rings among RLSEs are characterized. Then we establish how RLSEs correspond to lattice-ordered algebras of numerical events. Further, functions for associating lattice-ordered algebras of $S$-probabilities to RLSEs are studied. It is shown that there are only two ways to assign lattice-ordered algebras of $S$-probabilities to RLSEs if one restricts the corresponding mappings to term functions over the underlying orthomodular lattice. These term functions are the very functions by which also Boolean algebras can be assigned to Boolean rings. 
\end{abstract}



\section{Introduction}

In axiomatic quantum mechanics orthomodular posets and in particular orthomodular lattices often serve as ``quantum logics'' which are considered characteristic for the behaviour of a physical system. If a quantum logic turns out to be a Boolean algebra then one will have reason to assume that one deals with a classical physical system. The elements of a quantum logic can also be interpreted as events, and a Boolean algebra then as the equivalent of a classical field of events as known from probability theory. Moreover, Boolean algebras are in one-to-one correspondence with Boolean rings, and these and generalizations of them will provide a suitable setting for sets of so called numerical events which we will study in this paper. First of all we explain what is meant by a numerical event.

Let $S$ be a set of states of a physical system and $p(s)$ the probability of the occurrence of an event when the system is in state $s\in S$. The function $p$ from $S$ to $[0,1]$ is called a {\em numerical event}, or alternatively more precisely an {\em $S$-probability} (cf.\ \cite{BM91}, \cite{BM93} and \cite{MT}).

Considering sets $P$ of $S$-probabilities including the constant functions $0$ and $1$ we assume that every $P$ is partially ordered by the order $\leq$ of functions. If the supremun or infimum of two elements $p,q$ exists within $P$, we will denote it by $p\vee q$ and $p\wedge q$, respectively. Further, $p+q$, $p-q$ and $np$ for $n\in\mathbb N$ shall denote the sum, difference and $n$-fold of $p$ and $q$ considered as real functions, respectively. Moreover, we will write $p'$ for $1-p$, and if $p\leq q'$, in which case $p$ and $q$ are called {\em orthogonal} to each other, we will indicate this by $p\perp q$. Next we suppose that such sets $P$ will give rise to the following structure.

\begin{definition}\label{ASP}
{\rm(}cf.\ {\rm\cite{BM91}} and {\rm\cite{MT})} An {\em algebra of $S$-probabilities} is a set $P$ of $S$-probabilities such that
\begin{enumerate}
\item[{\rm(A1)}] $0,1\in P$,
\item[{\rm(A2)}] $p'\in P$ for every $p\in P$,
\item[{\rm(A3)}] if $p\perp q\perp r\perp p$ for $p,q,r\in P$ then $p+q+r\in P$.
\end{enumerate}
\end{definition}

We point out that by putting $r=0$ in axiom (A3) one obtains: If $p\perp q$ then $p+q\in P$ and, as can be easily verified, in this case $p+q=p\vee q$.

If we assume that {\em $P$ is a lattice}, what we will do from now on, then $'$ is an orthocomplementation in this lattice and $(p\wedge q')+(p'\wedge q)=(p\wedge q')\vee(p'\wedge q)$ because $(p\wedge q')\perp(p'\wedge q)$ due to De Morgan's laws which hold in any algebra of $S$-probabilities. If $P$ is a Boolean algebra and one defines $p+_1 q:=(p\wedge q')+(p'\wedge q)$ and $pq:=p\wedge q$ one obtains a Boolean ring $(P,+_1,\cdot,0,1)$. In this case obviously $p+_11=p'=1-p$, and $p\perp q\perp r\perp p$ means that $p\wedge q=q\wedge r=r\wedge p=0$ from which one obtains that $p\vee q\vee r=p+_1q+_1r\in P$ in accordance with the above axioms which can be motivated this way.

That an algebra of $S$-probabilities is a lattice and hence an orthomodular lattice (because any algebra of $S$-probabilities is an orthomodular poset, cf.\ \cite{MT}) is a typical feature of Hilbert-space logics:

Let $H$ be a Hilbert space, $P(H)$ the set of orthogonal projectors of $H$, $S$ the set of one-dimensional subspaces of $H$, and for every $s\in S$ let $v_s$ be a fixed vector belonging to $s$. Then, denoting the inner product in $H$ by $\langle.,.\rangle$, the set of functions $\{s\mapsto\langle Av_s,v_s\rangle\mid A\in P(H)\}$ is an algebra of $S$-probabilities which is isomorphic to the lattice of closed subspaces of $H$ (cf.\ \cite{BM93}).

A crucial question is to find out whether a lattice-ordered algebra of $S$-probabilities $P$ is a Boolean algebra and hence corresponds to a Boolean ring. If this is the case the multiplication in this ring is given by $pq=p\wedge q$ and the addition by a term equivalent to $(p\wedge q')\vee(p'\wedge q)$. However, one wants to express the addition of Boolean rings preferably by the operation $+$ of functions of $P$, as above with $+_1$ or as in the following theorem.

\begin{theorem}\label{th1}
Let $P$ be an algebra of $S$-probabilities which is a lattice. Then $P$ is a Boolean algebra if and only if $p+q-2(p\wedge q)\le1$ for all $p,q\in P$. If this is the case the addition in the corresponding Boolean ring can be expressed by $p\widehat+q:=p+q-2(p\wedge q)$. 
\end{theorem}

\begin{proof}
Assume $P$ to be a Boolean algebra. Then
\[
p+q-2(p\wedge q)=(p-(p\wedge q))+(q-(p\wedge q))=(p\wedge(p\wedge q)')+(q\wedge(p\wedge q)'),
\]
because for any algebra $A$ of $S$-probabilities $x\ge y$ for $x,y\in A$ implies $x-y=x\wedge y'$ (cf.\ e.g.\ \cite{DDL}). Taking into account that $p\wedge(p\wedge q)'\perp q\wedge(p\wedge q)'$ as one can see immediately by comparing $p\wedge(p\wedge q)'=p\wedge q'$ and $(q\wedge(p\wedge q)')'= p\vee q'$ we therefore obtain $p+q-2(p\wedge q)=(p\wedge(p\wedge q)')\vee(q\wedge(p\wedge q)')\in P$ and hence $p+q-2(p\wedge q)\le1$. Conversely, if $p+q-2(p\wedge q)\le1$ then $p+q\le1$ if $p\wedge q=0$. This means that $p\wedge q=0$ implies $p\perp q$, a property known with partial orders as being Boolean (cf.\ \cite{MT} -- \cite{T91}) -- not to be mixed up with being a Boolean algebra. However, for ortholattices being Boolean actually means to be a Boolean algebra (cf.\ \cite{T91}). So we deal with a Boolean algebra. As already shown above, $p\widehat+q=(p\wedge(p\wedge q)')\vee(q\wedge(p\wedge q)')$. By means of a few term-transformations one can see that this expression equals $p+_1q$.
\end{proof}

If a lattice-ordered algebra of $S$-probabilities does not give rise to a Boolean ring, which in general will be the case, one will expect that it will correspond to some kind of structure generalizing a Boolean ring. This leads us to the following definition.

\begin{definition}\label{def1}
A {\em ring-like structure of events (RLSE)} is an algebra $(R,\oplus,\cdot,0,1)$ of type $(2,2,0,0)$ such that $(R,\cdot,1)$ is an idempotent commutative monoid with zero-element $0$ {\rm(}i.e.\ $x0\approx0${\rm)} satisfying the following identities:
\begin{enumerate}[{\rm(R1)}]
\item $x\oplus y\approx y\oplus x$,
\item $(xy\oplus1)(x\oplus1)+1\approx x$,
\item $((xy\oplus1)x\oplus1)x\approx xy$,
\item $xy\oplus(x\oplus1)\approx(xy\oplus1)x\oplus1$.
\end{enumerate}
\end{definition}

The Boolean rings considered above are examples of RLSEs as one can easily verify by defining $pq:=p\wedge q$, taking for $\oplus$ the additions $+_1$ or $\widehat+$, respectively and considering $0$ and $1$ of the Boolean rings as the nullary operations of the corresponding RLSE.

Moreover, any lattice-ordered algebra of $S$-probabilities $P$ constitutes an RLSE if one sets $pq:=p\wedge q$ and $p\oplus q:=(p\wedge q')\vee(p'\wedge q)$. In  this case axiom (R1) is clearly fulfilled, (R2) translates into the absorption law of lattices, (R3) reflects the orthomodularity of $P$ and (R4) can straightforward be verified. The question arises which other choices of $\oplus$ are there to associate algebras of $S$-probabilities to RLSEs or special subclasses of them, and, on the other hand, given an RLSE $R$, under what conditions can one find a lattice-ordered algebra of $S$-probabilities which corresponds to $R$. We will provide answers to both problems.

Our research is to some extent related to the study of the one-to-one correspondence of bounded lattices with an involutory antiautomorphism and pGBQRs (partial generalized Boolean quasirings). Like with RLSEs a pGBQR $\mathbf R$ is an idempotent commutative monoid with a zero-element in which a partial operation $\oplus:\{0,1\}\times R\rightarrow R$ is defined which satisfies the correspondingly modified axiom (R2) from above and in addition satisfies the identity $0\oplus x\approx x$ which, as we will see, also holds in an RLSE. For results within this general framework of pGBQRs cf.\ \cite{BDM} and \cite{DDL10} -- \cite{DLM01}.

\section{Ring-like structures of events}

First, we agree that for elements $x,y$ of an RLSE $(R,\oplus,\cdot,0,1)$ $\leq$ denotes the partial order relation induced by the meet-semilattice $(R,\cdot)$ and that  $x$ and $y$ shall be called {\em orthogonal} -- in symbols $x\perp y$ -- if $x\leq y'$.

For every algebra $\mathbf R=(R,\oplus,\cdot,0,1)$ of type $(2,2,0,0)$ put
\begin{align*}
       x' & :=x\oplus1,\\
  x\vee y & :=(x'y')', \\
x\wedge y & :=xy      
\end{align*}
for all $x,y\in R$ and $\mathbb L(\mathbf R):=(R,\vee,\wedge,{}',0,1)$.

\begin{theorem}\label{th2}
Let $\mathbf R=(R,\oplus,\cdot,0,1)$ be an algebra of type $(2,2,0,0)$. Then $\mathbf R$ is an {\rm RLSE} if and only if $\mathbb L(\mathbf R)$ is an orthomodular lattice and $\oplus$ satisfies the following conditions:
\begin{enumerate}[{\rm(i)}]
\item $x\oplus y\approx y\oplus x$,
\item $x\oplus 1\approx x'$, 
\item $x\oplus y=x\vee y$ for all $x,y\in R$ with $x\leq y'$.
\end{enumerate}
\end{theorem}

\begin{proof}
First assume $\mathbf R$ to be an RLSE. Then $(R,\wedge)$ is a meet-semilattice. Let $\leq$ denote the corresponding partial order relation on $R$. Then $(R,\leq,0,1)$ is a bounded poset and $x\wedge y=\inf(x,y)$ for all $x,y\in R$. Now (R2) can be written in the form
\begin{enumerate}
\item[(R2')] $((x\wedge y)'\wedge x')'\approx x$.
\end{enumerate}
Putting $y=1$ yields $x''\approx x$, i.e.\ $'$ is an involution. Hence (R2') is equivalent to
\begin{enumerate}
\item[(R2'')] $(x\wedge y)'\wedge x'\approx x'$,
\end{enumerate}
i.e.\ $'$ is antitone. Since $'$ is an antitone involution on $(R,\leq)$, we have $x\vee y=\sup(x,y)$ for all $x,y\in R$ and $(R,\vee,\wedge,0,1)$ is a bounded lattice. Now (R3) and (R4) can be rewritten in the form
\begin{enumerate}
\item[(R3')] $((x\wedge y)\vee x')\wedge x\approx x\wedge y$,
\item[(R4')] $(x\wedge y)\oplus x'\approx(x\wedge y)\vee x'$,
\end{enumerate}
respectively. (R3') says that $\mathbb L(\mathbf R)$ satisfies the orthomodular law. Putting $y=0$ in (R3') yields $x\wedge x'\approx0$ which implies $x\vee x'\approx1$ by De Morgan's laws. This shows that $\mathbb L(\mathbf R)$ is an orthomodular lattice. Now (i) follows from (R1), (ii) from the definition of ${}'$ and (iii) from (R4'). Conversely, assume that $\mathbb L(\mathbf R)$ is an orthomodular lattice and $\oplus$ satisfies (i) -- (iii). Then $(R,\cdot,1)$ is a commutative idempotent monoid satisfying the identity $x0\approx0$. Since $'$ is an antitone involution on $(R,\leq)$, identities (R2) -- (R4) are equivalent to identities (R2''), (R3') and (R4'), respectively. Thus (R2'') follows since $'$ is antitone, (R3') is the orthomodular law and (R4') ensues from (iii).
\end{proof} 

\begin{corollary}\label{cor1}
Let $(R,\oplus,\cdot,0,1)$ be an {\rm RLSE}. Then the following hold for $x,y\in R$:
\begin{enumerate}[{\rm(a)}]
\item $(x\oplus1)\oplus1\approx x$,
\item $x(x\oplus1)\approx0$, and as a consequence $1\oplus1\approx0$,
\item $x\oplus0\approx x$,
\item $x\oplus(x\oplus1)\approx1$,
\item $x\le y$ if and only if $y\oplus1\le x\oplus1$.
\end{enumerate}
\end{corollary}

\begin{remark}\label{R4}
{\rm(R4)} is equivalent to
\begin{enumerate}
\item[{\rm(R4'')}] $x,y\in R$ and $x\perp y$ imply $(x\oplus y)\oplus1=(x\oplus1)(y\oplus1)$.
\end{enumerate}
\end{remark}

An RLSE $\mathbf R$ is called {\em weakly associative}, if it satisfies the identity $(x\oplus y)\oplus1\approx x\oplus(y\oplus1)$ (cf.\ the analogous definition for generalized Boolean quasirings in \cite{DLM01}).
 
\begin{theorem}\label{BR}
An {\rm RLSE} $\mathbf R=(R,\oplus,\cdot,0,1)$ is a Boolean ring if and only if it is weakly associative and satisfies the identity
\begin{enumerate}
\item[{\rm(T)}] $(xy'\oplus1)(x'y\oplus1)\oplus1\approx x\oplus y$.
\end{enumerate}
\end{theorem}

\begin{proof}
Identity (T) can be rewritten in the form
\begin{enumerate}
\item[(T')] $(x\wedge y')\vee(x'\wedge y)\approx x\oplus y$.
\end{enumerate}
If $\mathbf R$ is a Boolean ring then it is weakly associative and due to distributivity satisfies (T). Conversely, if $\mathbf R$ is weakly associative and satisfies (T) and hence (T') then  
\begin{align*}
c(x,y):=(x\wedge y)\vee(x\wedge y')\vee(x'\wedge y)\vee(x'\wedge y')\approx\\ \approx((x\wedge y')\vee(x'\wedge y))\vee((x\wedge y)\vee(x'\wedge y'))\approx \\
                                                             \approx(x\oplus y)\vee(x\oplus y')\approx(x\oplus y)\vee(x\oplus y)'\approx1.
\end{align*}
As it is well known, $c(x,y)\approx1$ implies that all elements pairwise commute which means that $\mathbb L(\mathbf R)$ is a Boolean algebra. Identity (T') guarantees that this Boolean algebra is linked to a Boolean ring, namely $\mathbf R$.
\end{proof}

We point out that for an RLSE $\mathbf R$ the orthomodular lattice $\mathbb L(\mathbf R)$ can be a Boolean algebra without $\mathbf R$ being a Boolean ring, as the following example shows.

\begin{example}
If one defines a binary operation $\oplus$ on $2^{\{1,2\}}$ by
\[
A\oplus B:=\left\{
\begin{array}{ll}
\{1,2\}                  & \text{if }A=B\text{ and }|A|=1 \\
(A\cap B')\cup(A'\cap B) & \text{otherwise}
\end{array}
\right.
\]
for all $A,B\in2^{\{1,2\}}$ then $(2^{\{1,2\}},\oplus,\cap,\emptyset,\{1,2\})$ is an {\rm RLSE}, but not a Boolean ring since
\[
\{1\}\oplus\{1\}=\{1,2\}\neq\emptyset=\emptyset\cup\emptyset=(\{1\}\cap\{1\}')\cup(\{1\}'\cap\{1\}).
\]
\end{example}

Next we study the closeness of RLSEs to algebras of $S$-probabilities.

\begin{theorem}\label{RA}   
Let $\mathbf R=(R,\oplus,\cdot,0,1)$ be an {\rm RLSE} {\rm(}partially ordered by $\le${\rm)} and $Q$ a set of $S$-probabilities {\rm(}partially ordered by $\le$ of functions{\rm)} which is order-isomorphic to $\mathbf R$. Further, let $f$ denote this order-isomorphism from $(R,\leq)$ to $(Q,\leq)$. If $f(x\oplus y)=f(x)+f(y)$ for all $x,y\in R$ with $x\perp y$ then $\mathbf R$ is isomorphic to a lattice-ordered algebra of $S$-probabilities, and any lattice-ordered algebra of $S$-probabilities can be obtained this way.
\end{theorem}

\begin{proof}
Let $a,b,c\in R$. Since $f$ is an order-isomorphism, it is also a lattice isomorphism. Because $f(0)+f(0)=f(0\oplus0)=f(0\vee0)=f(0)$ we have $f(0)=0$. Now
\[
f(a)+f(a')=f(a\oplus a')=f(a\vee a')=f(1)=1
\]
whence $f(a')=1-f(a)=(f(a))'$. This shows $f(1)=f(0')=(f(0))'=0'=1$ and that $a\perp b$ if and only if $f(a)\perp f(b)$. Next we check (A1) -- (A3).
\begin{enumerate}
\item[(A1)] $0=f(0)\in Q$ and $1=f(1)\in Q$.
\item[(A2)] $(f(a))'=f(a')\in Q$.
\item[(A3)] If $f(a)\perp f(b)\perp f(c)\perp f(a)$ then $a\perp b\perp c\perp a$ and hence
\begin{align*}
f(a)+f(b)+f(c) & =(f(a)+f(b))+f(c)=f(a\oplus b)+f(c)=f(a\vee b)+f(c)= \\
               & =f((a\vee b)\oplus c)\in Q.
\end{align*}
\end{enumerate}
This confirms that $Q$ is a lattice-ordered algebra of $S$-probabilities. Conversely, let $Q_1$ be a lattice-ordered algebra of $S$-probabilities. Then $\mathbf Q_1:=(Q_1,\vee,\wedge,{}',0,1)$ is an orthomodular lattice. Put $p\oplus q:=(p\wedge q')\vee(p'\wedge q)$ for all $p,q\in Q_1$. Then $\mathbf R_1:=(Q_1,\oplus,\wedge,0,1)$ is an RLSE with $\mathbb L(\mathbf R_1)=\mathbf Q_1$. Obviously, the poset corresponding to $Q_1$ coincides with the poset corresponding to $\mathbf R_1$.
\end{proof}

To make sure that an order-isomorphism as stated in Theorem~\ref{RA} exists we first translate the well known notion of a state of a lattice to RLSEs:

A {\em state} on an RLSE $\mathbf R=(R,\oplus,\cdot,0,1)$ is a mapping $m$ from $R$ to $[0,1]$ such that $m(1)=1$ and $m(x\oplus y)=m(x)+m(y)$ for all $x,y\in R$ with $x\perp y$. Further, a set $S$ of states of $\mathbf R$ is said to be {\em full} if for $x,y\in R$ we have $x\le y$ if and only if $m(x)\leq m(y)$ for all $m\in S$. Of course, a state on $\mathbf R$ is the same as a state on the orthomodular lattice $\mathbb L(\mathbf R)$ according to Theorem~\ref{th2}.

Now, given a full set $S$ of states on an RLSE $\mathbf R$ and defining for every $x\in R$ a function $q_x:S\rightarrow[0,1]$ by $q_x(m):=m(x)$ for all $m\in S$, one obtains that the set $\{q_x\mid x\in R\}$ becomes an algebra of $S$-probabilities, as one knows from the more general case of orthomodular posets admitting a full set of states which are in one-to-one correspondence to algebras of $S$-probabilities (cf.\ \cite{MT}; in this paper the notion algebra of $S$-probabilities is not yet explicitly used). In view of Theorem~\ref{th2} we can hence observe

\begin{theorem}
An {\rm RLSE} is order-isomorphic to a set $P$ of $S$-probabilities if it admits a full set of states, in which case $P$ is a lattice-ordered algebra of $S$-probabilities.
\end{theorem}

\section{Associating algebras of $S$-probabilities to RLSEs}
 
As already mentioned, any lattice-ordered algebra of $S$-probabilities $P$ can be turned into an RLSE, such that it is order-isomorphic to $P$ (see Theorem~\ref{RA}) and will then have a full set of states by what was said above. However, there will be many ways to align lattice-ordered algebras of $S$-probabilities to RLSEs:

\begin{remark}
Given a lattice-ordered algebra of $S$-probabilities $P$. If one defines $pq:=p\wedge q$ and $p\oplus q$ in such a way that the identities $p\oplus1\approx p'$, {\rm(R1)} and {\rm(R4)} hold for $P$, then $(P,\oplus,\cdot,0,1)$ becomes an {\rm RLSE} with the property that $p\oplus q=p+q$ for all $p,q\in P$ with $p\perp q$.
\end{remark}

From now on, when associating a lattice-ordered algebra of $S$-probabilities $P$ to an RLSE $(P,\oplus,\cdot,0,1)$ we will always assume that $P$ satisfies the identities $pq\approx p\wedge q$ and $p\oplus1\approx p'$. This then leaves to study operations $\oplus$ which are commutative and satisfy identity (R4).

As for the choice of the operation $\oplus$ we require that $x\oplus y$ should be representable by a binary term $t(x,y)$ in the operations $\vee,\wedge,{}'$ of $P$.

Examples for such terms $x\oplus y=t(x,y)$ satisfying identities (R1) -- (R4) are
\begin{align*}
       t_1(x,y) & :=(x'\wedge y)\vee(x\wedge y'), \\
\widehat t(x,y) & :=(x\wedge(x'\vee y'))\vee(y\wedge(x'\vee y')), \\
       t_2(x,y) & :=(x\vee y)\wedge(x'\vee y').
\end{align*}
As can be seen immediately $t_1(x,y)\le\widehat t(x,y)\le t_2(x,y)$. Further, due to the orthomodularity of $P$, one obtains that $\widehat t(x,y)=t_2(x,y)$.

We ask for all binary term functions $t(x,y)$ with $t_1(x,y)\le t(x,y)\le t_2(x,y)$ and first observe that only in case $P$ is a Boolean algebra $t_1(x,y)=t_2(x,y)$. Because, if $P$ is a Boolean algebra then, as a few calculations show, $t_1=t_2$, and conversely, if $t_1=t_2$ for all $x,y$ of a lattice-ordered algebra of $S$-probabilities $P$, then the function $c(x,y)$ (see the proof of Theorem~\ref{BR}) turns out to equal $t_1\vee t_2'=t_1\vee t_1'=1$ from which one can infer that $P$ is a Boolean algebra.

To see that $t_1\neq t_2$ the six-element algebra of $S$-probabilities with the four pairwise incomparable elements $a,a',b,b'$ between $0$ and $1$ can give an example: $t_1(a,b)=0$ whereas $t_2(a,b)=1$.

Next we turn our attention to finding all $t(x,y)$ between $t_1(x,y)$ and $t_2(x,y)$. For this end we agree on the restriction only to take into account those $t(x,y)$ which can be represented by binary terms in $\vee$, $\wedge$ and $'$ .

\begin{theorem}
For every orthomodular lattice $\mathbf L=(R,\vee,\wedge,{}',0,1)$ and hence, in particular, for every lattice-ordered algebra of $S$-probabilities, there exist exactly two {\rm RLSEs} $\mathbf R_i=(R,\oplus_i,\cdot,0,1)$ {\rm(}$i=1,2${\rm)} with $\mathbb L(\mathbf R_i)=\mathbf L$ such that $\oplus_i$ is a binary term in $\vee,\wedge,{}'$, namely
\begin{align*}
x\oplus_1y & :=(x\wedge y')\vee(x'\wedge y), \\
x\oplus_2y & :=(x\vee y)\wedge(x'\vee y').
\end{align*}
\end{theorem}

\begin{proof}
We use the description of the ninety-six binary terms over orthomodular lattices published in \cite N. Consider the following binary terms:
\[
\begin{array}{l|l|l|l}
i & t_i(x,y) & t_i(x,1) & x\leq y' \\
\hline
1 & x\wedge y                          & x  & 0 \\
2 & x\wedge y'                         & 0  & x \\
3 & x'\wedge y                         & x' & y \\
4 & x'\wedge y'                        & 0  & x'\wedge y' \\
5 & x\wedge(x'\vee y)\wedge(x'\vee y') & 0  & 0 \\
6 & x'\wedge(x\vee y)\wedge(x\vee y')  & 0  & 0 \\
7 & y\wedge(x\vee y')\wedge(x'\vee y') & 0  & 0 \\
8 & y'\wedge(x\vee y)\wedge(x'\vee y)  & 0  & 0
\end{array}
\]
Then all ninety-six binary terms over orthomodular lattices are given by $\bigvee\limits_{i\in I}t_i$ where $I$ is a subset of $\{1,\ldots,8\}$ satisfying $|I\cap\{5,\ldots,8\}|\in\{0,1,4\}$. We have
\[
\bigvee_{i\in I}t_i=(x\vee y)\wedge(x\vee y')\wedge(x'\vee y)\wedge(x'\vee y')
\]
for all $I\subseteq\{5,\ldots,8\}$ with $|I|>1$. Let $t$ be a binary term satisfying
\begin{enumerate}[(i)]
\item $t(x,y)\approx t(y,x)$,
\item $t(x,1)\approx x'$,
\item $t(x,y)=x\vee y$ in case $x\leq y'$.
\end{enumerate}
Then there exists some $I\subseteq\{1,\ldots,8\}$ with $|I\cap\{5,\ldots,8\}|\in\{0,1,4\}$ satisfying $\bigvee\limits_{i\in I}t_i=t$. Because of (iii), $\{2,3\}\subseteq I\subseteq\{1,\ldots,8\}\setminus\{4\}$ and because of (ii) $1\notin I$. This shows that $I$ is one of the following sets:
\[
\{2,3\},\{2,3,i\}\text{ with }i\in\{5,\ldots,8\},\{1,\ldots,8\}\setminus\{1,4\}.
\]
If $I=\{2,3,5\}$ then
\begin{align*}
t(x,y) & =(x\wedge y')\vee(x'\wedge y)\vee(x\wedge(x'\vee y)\wedge(x'\vee y'))\neq \\
       & \neq(x\wedge y')\vee(x'\wedge y)\vee(y\wedge(x\vee y')\wedge(x'\vee y'))=t(y,x)
\end{align*}
contradicting (i). Hence $I\neq\{2,3,5\}$. Because of symmetry reasons,
\[
I\neq\{2,3,6\},\{2,3,7\},\{2,3,8\}.
\]
Therefore $I=\{2,3\}$ or $I=\{1,\ldots,8\}\setminus\{1,4\}$ which means either
\[
t=(x\wedge y')\vee(x'\wedge y)
\]
or
\begin{align*}
t & =(x\wedge y')\vee(x'\wedge y)\vee((x\vee y)\wedge(x\vee y')\wedge(x'\vee y)\wedge(x'\vee y'))= \\
  & =(x\wedge y')\vee(x'\wedge y)\vee(((x'\vee y)\wedge(x\vee y'))\wedge((x\vee y)\wedge(x'\vee y')))=(x\vee y)\wedge(x'\vee y').
\end{align*}
\end{proof}

To make the mapping of lattice-ordered algebras of $S$-probabilities to RLSEs unique and not requiring from the outset that $\oplus$ should be a binary term function, one can take into account the class of RLSEs satisfying the following identity:
\begin{enumerate}
\item [(R5)] $x\oplus y\approx x(y\oplus1)\oplus(x\oplus1)y$
\end{enumerate}

\begin{remark}
There is only one possibility to assign a lattice-ordered algebra of $S$-pro\-ba\-bi\-li\-ties to an {\rm  RLSE} that satisfies identity {\rm(R5)} by a binary term function $t(x,y)$, namely by choosing $t(x,y)=x\oplus_1y:=(x\wedge y')\vee(x'\wedge y)$.
\end{remark}

\begin{proof}
Since $x\wedge y'\perp x'\wedge y$ we obtain
\[
t(x,y)=x\oplus y=(x\wedge y')\oplus(x'\wedge y)=(x\wedge y')\vee(x'\wedge y).
\]
\end{proof}

Authors' addresses:

Dietmar Dorninger \\
TU Wien \\
Faculty of Mathematics and Geoinformation \\
Institute of Discrete Mathematics and Geometry \\
Wiedner Hauptstra\ss e 8-10 \\
1040 Vienna \\
Austria \\
dietmar.dorninger@tuwien.ac.at

Helmut L\"anger \\
TU Wien \\
Faculty of Mathematics and Geoinformation \\
Institute of Discrete Mathematics and Geometry \\
Wiedner Hauptstra\ss e 8-10 \\
1040 Vienna \\
Austria, and \\
Palack\'y University Olomouc \\
Faculty of Science \\
Department of Algebra and Geometry \\
17.\ listopadu 12 \\
771 46 Olomouc \\
Czech Republic \\
helmut.laenger@tuwien.ac.at

\begin{thebibliography}{99}
\bibitem{BM91}
E.~G.~Beltrametti and M.~J.~M\c aczy\'nski, On a characterization of classical and nonclassical probabilities. J.\ Math.\ Phys.\ {\bf32} (1991), 1280--1286.
\bibitem{BM93}
E.~G.~Beltrametti and M.~J.~M\c aczy\'nski, On the characterization of probabilities: a generalization of Bell's inequalities. J.\ Math.\ Phys.\ {\bf34} (1993), 4919--4929.
\bibitem{BDM}
E.~G.~Beltrametti, D.~Dorninger and M.~J.~M\c aczy\'nski, On a cryptographical characterization of classical and nonclassical event systems.  Rep.\ Math.\ Phys.\ {\bf60} (2007), 117--123. 
\bibitem{DDL}
G.~Dorfer, D.~Dorninger and H.~L\"anger, On the structure of numerical event spaces. Kybernetica {\bf46} (2010), 971--981.
\bibitem{DDL10}
G.~Dorfer, D.~Dorninger and H.~L\"anger, On algebras of multidimensional probabilities. Math.\ Slovaca {\bf60} (2010), 571--582.
\bibitem{DLM97}
D.~Dorninger, H.~L\"anger and M.~M\c aczy\'nski, On ring-like structures occuring in axciomatic quantum mechanics. \"Osterreich.\ Akad.\ Wiss.\ Math.-Natur.\ Kl.\ Sitzungsber.\ II {\bf206} (1997), 279--289. 
\bibitem{DLM00}
D.~Dorninger, H.~L\"anger and M.~M\c aczy\'nski, Lattice properties of ring-like quantum logics. Internat.\ J.\ Theoret.\ Phys.\ {\bf39} (2000), 1015--1026. 
\bibitem{DLM01}
D.~Dorninger, H.~L\"anger and M.~M\c aczy\'nski, Ring-like structures with unique symmetric difference related to quantum logic. Discuss.\ Math.\ Gen.\ Algebra Appl.\ {\bf21} (2001), 239--253. 
\bibitem{MT}
M.~J.~M\c aczy\'nski and T.~Traczyk, A characterization of orthomodular partially ordered sets admitting a full set of states. Bull.\ Acad.\ Polon.\ Sci.\ S\'er.\ Sci.\ Math.\ Astronom.\ Phys.\ {\bf21} (1973), 3--8.
\bibitem N
M.~Navara, On generating finite orthomodular sublattices. Tatra Mt.\ Math.\ Publ.\ {\bf10} (1997), 109--117.
\bibitem{NP}
M.~Navara and P.~Pt\'ak, Almost Boolean orthomodular posets. J.\ Pure Appl.\ Algebra {\bf60} (1989), 105--111.
\bibitem{T91}
J.~Tkadlec, A note on distributivity in orthoposets. Demonstratio Math.\ {\bf24} (1991), 343--346.
\end{thebibliography}
\end{document}